\documentclass[a4paper,10pt]{amsart}
\usepackage{amsmath,amsfonts,amssymb,amscd,color,graphicx}
\usepackage[colorlinks]{hyperref}
\definecolor{linkblue}{RGB}{1,1,190}
\definecolor{citered}{RGB}{190,1,1}
\hypersetup{linkcolor=linkblue,urlcolor=linkblue,citecolor=citered}

\theoremstyle{plain}
\newtheorem{theorem}{\bf Theorem}[section]
\newtheorem{proposition}[theorem]{\bf Proposition}
\newtheorem{lemma}[theorem]{\bf Lemma}
\newtheorem{corollary}[theorem]{\bf Corollary}

\theoremstyle{definition}
\newtheorem{example}[theorem]{\bf Example}
\newtheorem{definition}[theorem]{\bf Definition}
\newtheorem{remark}[theorem]{\bf Remark}

\providecommand{\U}[1]{\protect\rule{.1in}{.1in}}
\numberwithin{equation}{section}

\begin{document}
\title{Multiplicative lattices with absorbing factorization}

\author{Andreas Reinhart}
\address{Institut f\"ur Mathematik und Wissenschaftliches Rechnen, Karl-Franzens-Universit\"at Graz, NAWI Graz, Heinrichstra{\ss}e 36, 8010 Graz, Austria}
\email{andreas.reinhart@uni-graz.at}

\author{G\"ul\c{s}en Ulucak}
\address{Department of Mathematics, Faculty of Science, Gebze Technical University, Gebze, Kocaeli, Turkey}
\email{gulsenulucak58@gmail.com, gulsenulucak@gtu.edu.tr}
	

\begin{abstract}
In \cite{YaNiNi21}, Yassine et al. introduced the notion of $1$-absorbing prime ideals in commutative rings with nonzero identity. In this article, we examine the concept of $1$-absorbing prime elements in C-lattices. We investigate the C-lattices in which every element is a finite product of $1$-absorbing prime elements (we denote them as OAFLs for short). Moreover, we study C-lattices having $2$-absorbing factorization (we denote them as TAFLs for short).
\end{abstract}

\maketitle

\section{Introduction}\label{1}

Let $L$ be a set together with an inner binary operation $\cdot$ on $L$ and a partial order $\leq$ on $L$ such that $(L,\cdot)$ is a monoid (i.e., $(L,\cdot)$ a commutative semigroup with identity) and $(L,\leq)$ is a complete lattice (i.e., each subset of $L$ has both a supremum and an infimum with respect to $\leq$). For each subset $E\subseteq L$, we let $\bigvee E$ denote the supremum of $E$, called the {\it join} of $E$ and we let $\bigwedge E$ denote the infimum of $E$, called the {\it meet} of $E$. For elements $a,b\in L$, let $a\vee b=\bigvee\{a,b\}$ and let $a\wedge b=\bigwedge\{a,b\}$. Moreover, set $1=\bigvee L$ and set $0=\bigwedge L$. We say that $(L,\cdot,\leq)$ is a {\it multiplicative lattice} if for all $x\in L$ and $E\subseteq L$, it follows that $1x=x$ and $x(\bigvee E)=\bigvee\{xe\mid e\in E\}$.

\medskip
We recall a few important situations in which multiplicative lattices occur. In what follows, we use the definitions of star operations, ideal systems and the specific star operations/ideal systems $v$, $t$ and $w$ without further mention. For more information on star operations see \cite{El19} and for more information on ideal systems see \cite{Ha98}. A profound introduction and study of the $w$-operation can be found in \cite{WaKi16}.

\begin{itemize}
\item It is well-known that if $R$ is a commutative ring with identity, $L$ is the set of ideals of $R$ and $\cdot:L\times L\rightarrow L$ is the ideal multiplication on $L$, then $(L,\cdot,\subseteq)$ is a multiplicative lattice.
\item Let $D$ be an integral domain and let $\ast$ be a star operation on $D$. Let $L$ be the set of $\ast$-ideals of $D$ together with the $\ast$-multiplication $\cdot_{\ast}:L\times L\rightarrow L$. Then $(L,\cdot_{\ast},\subseteq)$ is a multiplicative lattice.
\item Let $H$ be a commutative cancellative monoid and let $r$ be an ideal system on $H$. Let $L$ be the set of $r$-ideals of $H$ and let $\cdot_r:L\times L\rightarrow L$ be the $r$-multiplication. Then $(L,\cdot_r,\subseteq)$ is a multiplicative lattice.
\end{itemize}

\medskip
Let $L$ be a multiplicative lattice and let $e\in L$. For $a,b\in L$, we set $(a:b)=\bigvee\{x\in L\mid xb\leq a\}$. Then $e$ is called {\it weak meet principal} if $a\wedge e=(a:e)e$ for each $a\in L$ and $e$ is called {\it weak join principal} if $(be:e)=(0:e)\vee b$ for each $b\in L$. Furthermore, $e$ is said to be {\it meet principal} if $a\wedge be=((a:e)\wedge b)e$ for all $a,b\in L$ and $e$ is said to be {\it join principal} if $((a\vee be):e)=(a:e)\vee b$ for all $a,b\in L$. We say that $e$ is {\it weak principal} if $e$ is both weak meet principal and weak join principal. Finally, $e$ is said to be {\it principal} \cite{Di62} if $e$ is both meet principal and join principal. Observe that each principal element is weak principal. An element $a\in L$ is said to be {\it compact} if for each subset $F\subseteq L$ with $a\leq\bigvee F$, it follows that $a\leq\bigvee E$ for some finite subset $E$ of $F$. A subset $D\subseteq L$ is called {\it multiplicatively closed} if $1\in D$ and $xy\in D$ for each $x,y\in D$. A multiplicative lattice $L$ is called a {\it C-lattice} if $L$ is generated under joins by a multiplicatively closed subset $C$ of compact elements. Note that a finite product of compact elements in a C-lattice is again compact. By $L_*$ we denote the set of all compact elements of $L$. We say that $L$ is {\it principally generated} if every element of $L$ is the join of a set of principal elements of $L$. It is well-known (see \cite[Theorem 1.3]{An76}) that each principal element of a C-lattice is compact. Moreover, $L$ is said to be {\it join-principally generated} if each element of $L$ is the join of a set of join principal elements of $L$.

\medskip
Let $R$ be a commutative ring with identity, let $D$ be an integral domain, let $H$ be a commutative cancellative monoid, let $\ast$ be a star operation on $D$ and let $r$ be an ideal system on $H$. Note that the lattice of ideals of $R$ is a principally generated C-lattice. The lattice of $\ast$-ideals of $D$ is a C-lattice if and only if $\ast$ is a star operation of finite type. In analogy, it follows that the lattice of $r$-ideals of $H$ is a C-lattice if and only if $r$ is a finitary ideal system. Observe that the lattice of $v$-ideals of $D$ (or of $H$) can fail to be a C-lattice. Also note that even if $\ast$ is of finite type (resp. $r$ is finitary), then the lattice of $\ast$-ideals of $D$ (resp. the lattice of $r$-ideals of $H$) need not be (join-)principally generated. For instance, the $t$-operation is of finite type (resp. the $t$-system is finitary), but the lattice of $t$-ideals is (in general) not (join-)principally generated. We also want to emphasize that the lattice of $w$-ideals of $D$ (resp. of $H$) is a principally generated C-lattice.

\medskip
An element $a\in L$ is said to be {\it proper} if $a<1$, it is called {\it nilpotent} if $a^n=0$ for some $n\in\mathbb{N}$ and it is called {\it comparable} if $a\leq b$ or $b\leq a$ for each $b\in L$. For each $a\in L$, $L/a=\{b\in L\mid a\leq b\}$ is a multiplicative lattice with the multiplication $c\circ d=(cd)\vee a$ for elements $c,d\in L/a$. A proper element $p\in L$ is called {\it prime} if $ab\leq p$ implies $a\leq p$ or $b\leq p$ for all $a,b\in L$. Let ${\rm Spec}(L)$ denote the set of all prime elements of $L$. A proper element $m\in L$ is said to be {\it maximal} in $L$ if for each $x\in L$, $m<x\leq 1$ implies $x=1$. One can easily see that maximal elements are prime. For each $a\in L$, let $\min(a)$ be the set of prime elements of $L$ that are minimal above $a$. The lattice $L$ is called a {\it lattice domain} if $0$ is a prime element. ${\rm J}(L)$ is defined as the meet of all maximal elements of $L$. For $a\in L$, we define $\sqrt{a}=\bigwedge\{p\in L\mid p$ is prime and $a\leq p\}$. Note that in a C-lattice $L$, $\sqrt{a}=\bigwedge\{p\in L\mid a\leq p$ is a minimal prime over $a\}=\bigvee\{x\in L_*\mid x^n\leq a$ for some $n\in\mathbb{N}\}$. A proper element $q\in L$ is called {\it primary} if $ab\leq q$ implies $a\leq q$ or $b\leq\sqrt{q}$ for every $a,b\in L$. It is well-known that C-lattices can be localized at arbitrary multiplicatively closed subsets $S$ of compact elements as follows. The localization of $a\in L$ at $S$ is defined as $a_S=\bigvee\{x\in L\mid xs\leq a$ for some $s\in S\}$. The multiplication on $L_S=\{a_S\mid a\in L\}$ is defined by $c\circ_S d=(cd)_S$ for all $c,d\in L_S$. Let $p\in L$ be a prime element and let $S=\{x\in L_*\mid x\nleq p\}$. Then the set $S$ is a multiplicatively closed subset of $L$. In this case, the localization $L_S$ is denoted by $L_p$. It is well-known that $(L_p)_*=\{a_p\in L_p\mid a\in L_*\}$. Using this, it can be shown that if $L$ is a (principally generated) C-lattice, then $L_p$ is also a (principally generated) C-lattice for any prime element $p\in L$ (see \cite[Theorem 2.9]{An74}). It can also be proved that in a C-lattice $L$, for all $a,b\in L$, $(ab)_m=(a_mb_m)_m$ for each maximal element $m\in L$ and also, $a=b$ if and only if $a_n=b_n$ for all maximal elements $n\in L$. For more information on localizations, see \cite{An74,An76,Di62,JaJo95}.

\medskip
In \cite{YaNiNi21}, the authors introduced the concept of $1$-absorbing prime ideals in commutative rings with identity. These ideals are generalizations of prime ideals and many authors studied them from different points of view (see \cite{BoTaTeKo22}). The first aim of this paper is to study $1$-absorbing prime elements in C-lattices. Another (well-known) generalization of $1$-absorbing prime ideals are $2$-absorbing ideals. They have first been mentioned in \cite{Ba07} and in \cite{JaTeYe14}, the authors introduced $2$-absorbing elements in multiplicative lattices. A proper element $p\in L$ is called a {\it $2$-absorbing element} (also called a TA-element throughout this paper) if $abc\leq p$ implies that $ab\leq p$ or $bc\leq p$ or $ac\leq p$ for all $a,b,c\in L$.

The aforementioned concepts are part of the more general definition, namely that of $n$-absorbing ideals. These types of ideals were introduced and studied by Anderson and Badawi (see \cite{AnBa11}). It turns out that $n$-absorbing ideals are not just interesting objects in multiplicative ideal theory, but also in factorization theory. For instance, there is an important connection between $n$-absorbing ideals and the $\omega$-invariant in factorization theory (see \cite{AnBa11}). For a profound discussion of the $\omega$-invariant, we refer to \cite{GeHa08}.

We want to emphasize that the commutative rings in which each ideal is a finite product of $1$-absorbing prime ideals (resp. $2$-absorbing ideals, resp. $n$-absorbing ideals) have already been studied (see \cite{AhDuKh20,ElIsMaRe21,JaTeYe14}). The main goal of this paper is to consider principally generated C-lattices in which various types of elements can be written as finite products of $1$-absorbing prime elements or $2$-absorbing elements.

\medskip
We continue with a few more basic definitions that will be needed in the sequel. $L$ is said to be a {\it field} if $L=\{0,1\}$ and $L$ is called a {\it quasi-local} lattice if $1$ is compact and $L$ has a unique maximal element. The dimension of $L$, denoted by $\dim(L)$, is defined to be $\sup\{n\in\mathbb{N}\mid$ there exists a strict chain of prime elements of $L$ of length $n\}$. If $\dim(L)=0$, then $L$ is said to be a {\it zero-dimensional} lattice. Note that $L$ is a zero-dimensional lattice if and only if every prime element of $L$ is maximal. We say that a multiplicative lattice is {\it Noetherian} if every element of $L$ is compact (see \cite[page 352]{Ja03}). A multiplicative lattice is said to be a {\it Pr\"ufer} lattice if every compact element of $L$ is principal. (For more information about Pr\"ufer lattices, see \cite[Theorem 3.4]{An76}.) A {\it ZPI-lattice} is a multiplicative lattice in which every element is a finite product of prime elements \cite{Ja02}. A multiplicative lattice $L$ is said to be a {\it Q-lattice} if every element is a finite product of primary elements \cite{NaAl97}.

\medskip
Our paper is organized as follows. In Section~\ref{2}, we study the concept of $1$-absorbing prime elements (OA-elements). The relationships among prime elements, primary elements, and TA-elements are studied in Examples~\ref{Example 2.2} and~\ref{Example 2.3}. In Propositions~\ref{Proposition 2.6} and~\ref{Proposition 2.7} and Corollary~\ref{Corollary 2.9}, we demonstrate that the concepts of prime elements and OA-elements coincide in C-lattices that are not quasi-local. In Section~\ref{3}, we study C-lattices whose elements have a TA-factorization. We call a C-lattice a {\it TA-factorization lattice} (abbreviated as {\rm TAFL}) if every element possesses a TA-factorization. In Proposition~\ref{Proposition 3.4}, we show that $\dim(L)\leq 1$ if $L$ is a principally generated TAFL. In Theorem~\ref{Theorem 3.5}, we prove that a TAFL is a ZPI-lattice domain if it is a Pr\"ufer lattice domain. Then we study C-lattices for which all compact elements have a factorization into TA-elements, called {\it CTAFLs}. Finally, we explore the C-lattices for which all principal elements have a factorization into TA-elements, which we call {\it PTAFLs}. In Theorem~\ref{Theorem 3.11}, we show that if $(L,m)$ is a quasi-local principally generated C-lattice domain, then $L$ is a TAFL if and only if $L$ is a PTAFL and $\dim(L)\leq 1$. In Section~\ref{4}, we study the factorization lattices with respect to the OA-element concept, similar to Section~\ref{3}. We study {\it OA-factorization lattices} (abbreviated as {\rm OAFLs}) which are C-lattices in which every element possesses an OA-factorization. Then, we examine the C-lattices for which all compact elements have a factorization into OA-elements and we call them {\it COAFLs}. Finally, we explore the C-lattices for which all principal elements have a factorization into OA-elements, called {\it POAFLs}. Among the many results, in Theorem~\ref{Theorem 4.13}, we characterize OAFLs, COAFLs and lattices for which the join of any two principal elements has an OA-factorization. In Theorem~\ref{Theorem 4.13}, we also see that if $L$ is a principally generated OAFL, then it satisfies one of the following conditions.

\begin{enumerate}
\item[\textbf{i.}] $L$ is a ZPI-lattice.
\item[\textbf{ii.}] $L$ is a quasi-local lattice, $m^2$ is comparable and $m$ is a nilpotent element.
\item[\textbf{iii.}] $L$ is a quasi-local lattice domain, $m^2$ is comparable and $\bigwedge_{n\in\mathbb{N}} m^n=0$.
\end{enumerate}

In Theorem~\ref{Theorem 4.14}, we show that $L$ is a ZPI-lattice if and only if $L$ is a Pr\"ufer OAFL if and only if $L$ is a Pr\"ufer POAFL. Theorem~\ref{Theorem 4.15} establishes some relationships between OAFLs, COAFLs, POAFLs, TAFLs, CTAFLs and PTAFLs.

\section{On 1-absorbing prime elements of C-lattices}\label{2}

\begin{definition}\label{Definition 2.1}
Let $L$ be a C-lattice. A proper element $x\in L$ is called a {\it $1$-absorbing prime element} or an {\it OA-element} if for all $a,b,c\in L\setminus\{1\}$, $abc\leq x$ implies that $ab\leq x$ or $c\leq x$.
\end{definition}

It follows immediately from the definition that every OA-element is both a TA-element and a primary element. Moreover, every prime element is an OA-element. We infer that the class of OA-elements of $L$ lies between the classes of prime elements and TA-elements and also between the classes of prime elements and primary elements.

Let $L$ be a C-lattice and let $a\in L$. We obtain the following irreversible right arrows:

\begin{enumerate}
\item[(1)] $a$ is a prime element $\Rightarrow$ $a$ is an OA-element $\Rightarrow$ $a$ is a primary element.
\item[(2)] $a$ is a prime element $\Rightarrow$ $a$ is an OA-element $\Rightarrow$ $a$ is a TA-element.
\end{enumerate}

We give some examples to show that these arrows are not reversible.

\begin{example}\label{Example 2.2}
\textnormal{\{This example is inspired by \cite[Example 7]{DuEp22}\}}. Let $L$ be a C-lattice, having underlying set $\{0,1,a,b,c,d\}$ ordered by $a\leq b\leq d$ and $a\leq c\leq d$, with multiplication $xy=a$ for all $x,y\in\{a,b,c,d\}$. The prime elements of $L$ are $0$ and $d$. Moreover, $L$ is a quasi-local lattice. Note that $b$ is an OA-element of $L$ that is not a prime element. In particular, $b$ is a primary TA-element of $L$.
\end{example}

\begin{example}\label{Example 2.3}
We demonstrate that, in general, neither TA-elements nor primary elements are OA-elements. Let ${\rm I}(\mathbb{Z})$ be the lattice of ideals of $\mathbb{Z}$. Note that $(15)$ is a TA-element of ${\rm I}(\mathbb{Z})$ that is not an OA-element of ${\rm I}(\mathbb{Z})$. Furthermore, $(8)$ is a primary element of ${\rm I}(\mathbb{Z})$ that fails to be an OA-element of ${\rm I}(\mathbb{Z})$.
\end{example}

Next we provide a lemma whose second part was already proved in \cite[Lemma 1]{JaTeYe14}. For the sake of completeness, we include its proof.

\begin{lemma}\label{Lemma 2.4}
Let $L$ be a C-lattice and let $x\in L$ be proper.
\begin{enumerate}
\item[(1)] $x$ is an OA-element if and only if for all $a,b,c\in L_*\setminus\{1\}$, $abc\leq x$ implies that $ab\leq x$ or $c\leq x$.
\item[(2)] $x$ is a TA-element if and only if for all $a,b,c\in L_*$, $abc\leq x$ implies that $ab\leq x$ or $ac\leq x$ or $bc\leq x$.
\end{enumerate}
\end{lemma}

\begin{proof}
(1) ($\Rightarrow$) This is clear.

\smallskip
($\Leftarrow$) Let $a,b,c\in L\setminus\{1\}$ be such that $abc\leq x$ and $ab\nleq x$. We show that $c\leq x$ to complete the proof. Since $ab\nleq x$, there are some $a_1,b_1\in L_*$ such that $a_1\leq a$, $b_1\leq b$ and $a_1b_1\nleq x$. Let $c^{\prime}\in L_*$ be such that $c^{\prime}\leq c$. Since $a_1b_1c^{\prime}\leq abc\leq x$ and $a_1b_1\nleq x$, we infer that $c^{\prime}\leq x$. Consequently, $c\leq x$ (since $c$ is the join of the set of all $c_1\in L_*$ with $c_1\leq c$). Therefore, $x$ is an OA-element.

\medskip
(2) ($\Rightarrow$) This is obvious.

\smallskip
($\Leftarrow$) Let $a,b,c\in L$ be such that $abc\leq x$, $ab\nleq x$ and $ac\nleq x$. It remains to prove that $b^{\prime}c^{\prime}\leq x$ for all $b^{\prime},c^{\prime}\in L_*$ such that $b^{\prime}\leq b$ and $c^{\prime}\leq c$. Since $ab\nleq x$, there are some $a_1,b_1\in L_*$ such that $a_1\leq a$, $b_1\leq b$ and $a_1b_1\nleq x$. Since $ac\nleq x$, there are some $a_2,c_1\in L_*$ such that $a_2\leq a$, $c_1\leq c$ and $a_2c_1\nleq x$. Let $b^{\prime},c^{\prime}\in L_*$ be such that $b^{\prime}\leq b$ and $c^{\prime}\leq c$. Set $\overline{a}=a_1\vee a_2$, $\overline{b}=b_1\vee b^{\prime}$ and $\overline{c}=c_1\vee c^{\prime}$. Then $\overline{a},\overline{b},\overline{c}\in L_*$, $\overline{a}\overline{b}\nleq x$ and $\overline{a}\overline{c}\nleq x$. Since $\overline{a}\overline{b}\overline{c}\leq abc\leq x$, we infer that $b^{\prime}c^{\prime}\leq \overline{b}\overline{c}\leq x$. Consequently, $x$ is a TA-element.
\end{proof}

\begin{lemma}\label{Lemma 2.5}
Let $L$ be a C-lattice. Then, $u\vee w\neq 1$ for all distinct proper elements $u,w\in L$ if and only if $L$ is quasi-local.
\end{lemma}

\begin{proof}
($\Rightarrow$): Let $L$ be not quasi-local. There are two distinct maximal elements $m_1,m_2\in L$. Observe that $m_1\vee m_2=1$.

($\Leftarrow$): Let $L$ be quasi-local with maximal element $m$. Clearly, $u\vee w\leq m<1$ for all distinct proper elements $u,w\in L$.
\end{proof}

\begin{proposition}\label{Proposition 2.6}
Let $L$ be a C-lattice. If $L$ is not quasi-local, then each OA-element of $L$ is prime.
\end{proposition}

\begin{proof}
Suppose that $x$ is an OA-element of $L$ that is not a prime. It remains to show that $L$ is quasi-local. By the assumption, there are some $c,d\in L$ such that $cd\leq x$, $c\nleq x$ and $d\nleq x$. If $u\vee w\neq 1$ for all distinct proper elements $u,w\in L$, then we are done by Lemma~\ref{Lemma 2.5}. Assume that $u\vee w=1$ for two distinct proper elements $u,w\in L$. Since $wcd\leq x$ and $d\nleq x$, it follows that $wc\leq x$. Similarly, since $ucd\leq x$ and $d\nleq x$, we obtain that $uc\leq x$. Consequently, $uc\vee wc=(u\vee w)c\leq x$, and hence $c=1c=(u\vee w)c\leq x$, a contradiction. Therefore, $L$ is quasi-local.
\end{proof}

\begin{proposition}\label{Proposition 2.7}
Let $(L,m)$ be a quasi-local C-lattice and let $x\in L$ be proper. Then $x$ is an OA-element if and only if $x$ is a prime element or $m^2\leq x<m$.
\end{proposition}

\begin{proof}
($\Rightarrow$) Without restriction, we can assume that $x$ is not a prime element of $L$. Clearly, there are two proper elements $a,b\in L$ such that $ab\leq x$, $a\nleq x$ and $b\nleq x$. Set $y=m^2$. Note that $yab\leq ab\leq x$. Since $a$, $b$ and $y$ are proper elements of $L$ and $b\nleq x$, we have that $ya\leq x$, and hence $mma\leq x$. Moreover, since $a$ and $m$ are proper elements of $L$ and $a\nleq x$, this implies that $m^2=mm\leq x$. Since $x$ is not a prime element of $L$, it is obvious that $x<m$.

($\Leftarrow$) If $x$ is a prime element of $L$, then clearly $x$ is an OA-element of $L$. Now let $m^2\leq x<m$. Let $a,b,c\in L$ be proper such that $abc\leq x$ and $c\nleq x$. Note that $a\leq m$ and $b\leq m$, and hence $ab\leq m^2\leq x$. Therefore, $x$ is an OA-element.
\end{proof}

\begin{corollary}\label{Corollary 2.8}
Let $L$ be a C-lattice and let $x\in L$ be an OA-element of $L$.
\begin{enumerate}
\item[(1)] $\sqrt{x}$ is a prime element of $L$ with $(\sqrt{x})^2\leq x$.
\item[(2)] $(x:a)$ is a prime element of $L$ for each proper $a\in L$ with $a\nleq x$.
\end{enumerate}
\end{corollary}

\begin{proof}
(1) If $x$ is a prime element of $L$, then clearly $\sqrt{x}=x$ is a prime element of $L$ with $(\sqrt{x})^2=x^2\leq x$. Now let $x$ be not a prime element of $L$. It follows from Propositions~\ref{Proposition 2.6} and~\ref{Proposition 2.7} that $L$ is quasi-local with maximal element $m$ and $m^2\leq x<m$. We infer that $m=\sqrt{x}$, and hence $\sqrt{x}$ is a prime element of $L$ with $(\sqrt{x})^2=m^2\leq x$.
	
\medskip
(2) Let $b,c\in L$ be such that $bc\leq (x:a)$. Then $abc\leq x$. We obtain that $ab\leq x$ or $c\leq x$. Therefore, $b\leq (x:a)$ or $c\leq (x:a)$.
\end{proof}

As another consequence of Propositions~\ref{Proposition 2.6} and~\ref{Proposition 2.7}, we give the following corollary without proof.

\begin{corollary}\label{Corollary 2.9}
Let $L$ be a C-lattice. Then there is an OA-element of $L$ that is not prime if and only if $L$ is quasi-local with maximal element $m$ such that $m^2\neq m$.
\end{corollary}

\begin{proposition}\label{Proposition 2.10}
Let $L$ be a principally generated quasi-local Noetherian lattice with maximal element $m$. Then every OA-element is prime if and only if $L$ is a field.
\end{proposition}

\begin{proof}
($\Rightarrow$) Since every OA-element is prime, then we obtain that $m^2=m$. Therefore, $m=0$ by \cite[Theorem 1.4]{An76}, and thus $L$ is a field.

\medskip
($\Leftarrow$) This is clear.
\end{proof}

Now, we give a relation between OA-elements and lattice domains.

\begin{proposition}\label{Proposition 2.11}
Let $L$ be a C-lattice. Then $0$ is an OA-element of $L$ if and only if $L$ is a lattice domain or $L$ is quasi-local with maximal element $m$ such that $m^2=0$.
\end{proposition}

\begin{proof}
($\Rightarrow$) Let $0$ be an OA-element of $L$ and let $L$ be not a lattice domain. Then $0$ is not prime, and thus $L$ is quasi-local with maximal element $m$ by Proposition~\ref{Proposition 2.6}. We infer by Proposition~\ref{Proposition 2.7} that $m^2=0$.

($\Leftarrow$) This is an immediate consequence of Proposition~\ref{Proposition 2.7}.
\end{proof}

\begin{proposition}\label{Proposition 2.12}
Let $L$ be a principally generated C-lattice and set $m={\rm J}(L)$. The following statements are equivalent.
\begin{enumerate}
\item[(1)] Every proper element of $L$ is an OA-element.
\item[(2)] Every proper principal element of $L$ is an OA-element.
\item[(3)] $L$ is quasi-local and $m^2=0$.
\end{enumerate}
\end{proposition}

\begin{proof}
(1) $\Rightarrow$ (2) This is obvious.

\medskip
(2) $\Rightarrow$ (3) Assume that $L$ is not a quasi-local lattice. Then each proper principal element is a prime element. Note that $L$ is a lattice domain. Let $x\in L$ be a proper principal element. It follows that $x^2$ is a principal prime element. We conclude that $x=x^2$, and thus $1=x\vee (0:x)$. Since $x$ is proper, we have that $x=0$. Consequently, $L$ is a field. But this contradicts the fact that $L$ is not a quasi-local lattice. This implies that $L$ is quasi-local with maximal element $m$. We infer that $0$ is prime or $m^2=0$ by Proposition~\ref{Proposition 2.7}. Suppose that $m^2\neq 0$. Then there is a nonzero principal element $c\in L$ with $c\leq m^2$. Proposition~\ref{Proposition 2.7} implies that $c^2$ is prime element or $m^2\leq c^2$. If $c^2$ is prime, then $c^2=c$. If $m^2\leq c^2$, then $m^2\leq c^2\leq c\leq m^2$, and hence $c^2=c$. In any case, we obtain that $c^2=c$, and thus $1=c\vee (0:c)$, since $c$ is principal. Since $L$ is quasi-local, it follows that $c=1$, a contradiction. Therefore, $m^2=0$.

\medskip
(3) $\Rightarrow$ (1) This follows from Proposition~\ref{Proposition 2.7}.
\end{proof}

In general, in Proposition~\ref{Proposition 2.12}, statement (1) does not imply statement (3). To see this, we need to examine Example~\ref{Example 2.2}. We have that each proper element of $L$ in Example~\ref{Example 2.2} is an OA-element, and $L$ is a quasi-local C-lattice, but $m^2=a\neq 0$. Also note that the lattice $L$ in Example~\ref{Example 2.2} has no (join) principal elements except $0$ and $1$.

\begin{proposition}\label{Proposition 2.13}
Let $(L,m)$ be a quasi-local principally generated C-lattice such that $m^2$ is comparable. The following statements are equivalent.
\begin{enumerate}
\item[(1)] For each two principal elements $x,y\in L$ with $m^2\leq x$ and $m^2\leq y$, we have $x\leq y$ or $y\leq x$.
\item[(2)] $m$ is principal or $m=m^2$.
\item[(3)] If $a$ is an OA-element of $L$, then $a$ is prime or $a=m^2$.
\item[(4)] There are at most two elements between $m^2$ and $m$.
\end{enumerate}
\end{proposition}

\begin{proof}
(1) $\Rightarrow$ (2) Let $m\not=m^2$. There is some proper principal $x\in L$ with $x\nleq m^2$ (since $L$ is principally generated). We obtain that $m^2\leq x$. It remains to show that $x=m$. Assume that $x<m$. There is some proper principal $y\in L$ such that $y\nleq x$. Since $m^2\leq x$, we infer that $y\nleq m^2$, and hence $m^2\leq y$. This implies that $x<y$. Consequently, $x=yz$ for some proper $z\in L$, and thus $x\leq m^2$, a contradiction.

\medskip
(2) $\Rightarrow$ (3) Let $m$ be principal or $m=m^2$. Let $a\in L$ be an OA-element that is not a prime element. It follows from Proposition~\ref{Proposition 2.7} that $m^2\leq a<m$, and thus $m$ is principal. Observe that $a=bm$ for some proper $b\in L$. Clearly, $m^2\leq a=bm\leq m^2$. We conclude that $a=m^2$.

\medskip
(3) $\Rightarrow$ (4) Let $x\in L$ be principal such that $m^2\leq x$. Without restriction, let $x$ be proper. By Proposition~\ref{Proposition 2.7}, we get that $x$ is an OA-element of $L$. By the assumption, $x=m^2$. Therefore, there are at most two elements between $m^2$ and $m$.

\medskip
(4) $\Rightarrow$ (1) This is clear.
\end{proof}

Note that if $(L,m)$ is a quasi-local principally generated C-lattice such that $m$ is principal or $m=m^2$, then $m^2$ must be comparable. (The statement is clearly true if $m=m^2$. If $m$ is principal and $x\in L$ is such that $m^2\nleq x$, then $(x:m)\leq m$, and hence $x=m(x:m)\leq m^2$.)

\begin{proposition}\label{Proposition 2.14}
Let $L$ be a C-lattice and set $m={\rm J}(L)$. The following statements are equivalent.
\begin{enumerate}
\item[(1)] Every TA-element of $L$ is an OA-element of $L$.
\item[(2)] $L$ satisfies the following two conditions.
\begin{enumerate}
\item[(a)] ${\rm Spec}(L)$ is linearly ordered.
\item[(b)] If $x$ is a TA-element of $L$, then $\min(x)\subseteq\{x,m\}$.
\end{enumerate}
\item[(3)] $L$ is quasi-local and $\sqrt{x}=m$ for each nonprime TA-element $x\in L$.
\end{enumerate}
\end{proposition}

\begin{proof}
(1) $\Rightarrow$ (2) (a) Let $p$ and $q$ be prime elements of $L$. Then $p\wedge q$ is a TA-element of $L$. By assumption, $p\wedge q$ is an OA-element. Then by Corollary~\ref{Corollary 2.8}(1), we have that $\sqrt{p\wedge q}=p\wedge q$ is prime, and hence $p\wedge q=p$ or $p\wedge q=q$ by \cite[Lemma 7]{JaTeYe14}. We obtain that $p\leq q$ or $q\leq p$. Therefore, $L$ is quasi-local with maximal element $m$.

(b) Let $x\in L$ be a TA-element and let $p\in\min(x)$. If $x$ is prime, then it is clear that $x=p$. Now let $x$ be not a prime element. We infer by Proposition~\ref{Proposition 2.7} that $m^2\leq x\leq p\leq m$, and thus $p=m$.

\medskip
(2) $\Rightarrow$ (3) Suppose that $L$ satisfies (a) and (b). It is an immediate consequence of (a) that $L$ is quasi-local with maximal element $m$. Let $x\in L$ be a nonprime TA-element. Since $x$ is proper, there is some $p\in\min(x)$. Clearly, $x\not=p$, and thus $p=m$. Since $L$ is quasi-local, we have that $\min(x)=\{m\}$. This implies that $\sqrt{x}=m$.

\medskip
(3) $\Rightarrow$ (1) Let $x\in L$ be a TA-element. If $x$ is prime, then clearly $x$ is an OA-element of $L$. Now let $x$ be not prime. Then $\sqrt{x}=m$ and $m^2\leq x$ by \cite[Lemma 2]{JaTeYe14}. We infer by Proposition~\ref{Proposition 2.7} that $x$ is an OA-element of $L$.
\end{proof}

\begin{remark}\label{Remark 2.15}
Let $L$ be a C-lattice and let $r,x\in L$ be such that $r\leq x$.
\begin{enumerate}
\item[(1)] If $x$ is an OA-element of $L$, then $x$ is an OA-element of $L/r$.
\item[(2)] If $x$ is a TA-element of $L$, then $x$ is a TA-element of $L/r$.
\end{enumerate}
\end{remark}

\begin{proof}
(1) Let $x$ be an OA-element. Clearly, $x$ is a proper element of $L/r$. Let $a,b,c\in L/r$ be proper elements such that $a\circ b\circ c\leq x$. Then $abc\leq x$. By assumption, $ab\leq x$ or $c\leq x$, and hence $a\circ b\leq x$ or $c\leq x$. Consequently, $x$ is an OA-element of $L/r$.

(2) Let $x$ be a TA-element. Observe that $x$ is a proper element of $L/r$. Let $a,b,c\in L/r$ be such that $a\circ b\circ c\leq x$. We infer that $abc\leq x$, and thus $ab\leq x$ or $ac\leq x$ or $bc\leq x$. This implies that $a\circ b\leq x$ or $a\circ c\leq x$ or $b\circ c\leq x$. Therefore, $x$ is a TA-element of $L/r$.
\end{proof}

\begin{remark}\label{Remark 2.16}
Let $L$ be a C-lattice, let $S\subseteq L_*$ be multiplicatively closed and let $x\in L$ be such that $x_S\not=1$.
\begin{enumerate}
\item[(1)] If $x$ is an OA-element of $L$, then $x_S$ is an OA-element of $L_S$.
\item[(2)] If $x$ is a TA-element of $L$, then $x_S$ is a TA-element of $L_S$.
\end{enumerate}
\end{remark}

\begin{proof}
In what follows, we use without further mention that $(L_S)_*=\{a_S\mid a\in L_*\}$.

(1) Let $x$ be an OA-element. We apply Lemma~\ref{Lemma 2.4}(1). Let $a,b,c\in L_*$ be such that $a_S,b_S,c_S\not=1$ and $a_S\circ_S b_S\circ_S c_S\leq x_S$. Then $abc\leq x_S$, and hence $dabc\leq x$ for some $d\in S$. Since $a,b,c\not=1$, we have that $dab\leq x$ or $c\leq x$. Note that $d_S=1$, and thus $a_S\circ_S b_S\leq x_S$ or $c_S\leq x_S$. Therefore, $x_S$ is an OA-element of $L_S$.

(2) Let $x$ be a TA-element. We use Lemma~\ref{Lemma 2.4}(2). Let $a,b,c\in L_*$ be such that $a_S\circ_S b_S\circ_S c_S\leq x_S$. Then $abc\leq x_S$, and hence $dabc\leq x$ for some $d\in S$. Observe that $dab\leq x$ or $dac\leq x$ or $bc\leq x$. Since $d_S=1$, we have that $a_S\circ_S b_S\leq x_S$ or $a_S\circ_S c_S\leq x_S$ or $b_S\circ_S c_S\leq x_S$. Consequently, $x_S$ is a TA-element of $L_S$.
\end{proof}

Alternatively, we can use Propositions~\ref{Proposition 2.6} and ~\ref{Proposition 2.7} to prove Remark~\ref{Remark 2.16}(1).

\begin{theorem}\label{Theorem 2.17}
Let $L$ be a join-principally generated C-lattice. Every nonzero proper element of $L$ is an OA-element if and only if $L\cong L_1\times L_2$ where $L_1,L_2$ are fields or $L$ is quasi-local with maximal element $m$ such that $m=\sqrt{0}$ and $m^2\leq x$ for every nonzero proper join principal element $x\in L$. In particular, if these equivalent conditions are satisfied, then $\dim(L)=0$.
\end{theorem}

\begin{proof}
($\Rightarrow$) Let every nonzero proper element of $L$ be an OA-element. First let $L$ be quasi-local with maximal element $m$. By assumption, every nonzero proper element of $L$ is a TA-element. Now \cite[Theorem 8]{JaTeYe14} completes the proof. Now let $L$ be not quasi-local. Then the concepts of prime elements and OA-elements coincide. Let $m_1$ and $m_2$ be two distinct maximal elements of $L$. Assume that $m_1\wedge m_2\neq 0$. By the assumption, $m_1\wedge m_2$ is prime. It can be shown that $m_1=m_2$, a contradiction. It follows that $m_1\wedge m_2=0$, and thus $L\cong L/m_1\times L/m_2$. Note that $L/m_1, L/m_2$ are fields.

\medskip
($\Leftarrow$) If $L\cong L_1\times L_2$, where $L_1$ and $L_2$ are fields, then each nonzero proper element of $L$ is prime, and hence it is an OA-element. Now let $L$ be quasi-local with maximal element $m$ such that $m=\sqrt{0}$ and $m^2\leq x$ for every nonzero proper join principal element $x\in L$. Let $y$ be a nonzero proper element of $L$. There is some nonzero join principal element $c\in L$ with $c\leq y$. We have that $m^2\leq c\leq y$, and thus $y$ is an OA-element of $L$ by Proposition~\ref{Proposition 2.7}. It is clear that $\dim(L)=0$ in any case.
\end{proof}

Let $K$ be a field and let $L$ be the lattice of ideals of $K\times K$. Then $L$ is a join-principally generated C-lattice and each nonzero proper element of $L$ is an OA-element, but $L$ is not quasi-local. The examples and facts in the next paragraph were kindly provided to us by the referee.

\medskip
If $L$ is the lattice of ideals of $(\mathbb{Z}/2\mathbb{Z})[X,Y]/(X^2,Y^2)$, then $L$ is a quasi-local principally generated C-lattice for which every proper (principal) element is an OA-element. Moreover, if $L$ is the lattice of ideals of $\mathbb{Z}/8\mathbb{Z}$, then $L$ is a quasi-local principally generated C-lattice for which every nonzero proper (principal) element is an OA-element and yet $0$ is not an OA-element of $L$. Let $(L,m)$ be a quasi-local principally generated C-lattice such that $m=\sqrt{0}$ and $m^2\leq x$ for every nonzero proper principal element of $x\in L$. If $m^2=0$, then every proper (principal) element of $L$ is an OA-element by Proposition~\ref{Proposition 2.7}. Now let $m^2\not=0$. Clearly, there is a nonzero principal element $x\in L$ such that $x\leq m^2$, and hence $m^2=x$ is principal. Since $x$ is nonzero and principal and $L$ is quasi-local we have that $m^3\not=m^2$. (For if $m^3=m^2$, then $x^2=x$, which contradicts \cite[Theorem 1.4]{An76}.) In particular, we have that $m^3=0$ (since there is no nonzero principal element $y\in L$ with $y\leq m^3$).

\section[TAFLs]{TAFLs and their generalizations}\label{3}

In this section, we study C-lattices whose elements have a TA-factorization. A TA-factorization of an element $x\in L$ means that $x$ is written as a finite product of TA-elements $(x_k)_{k=1}^n$. (Note that the element $1$ is the empty product.) We say that a C-lattice $L$ is a {\it TA-factorization lattice} (abbreviated as {\rm TAFL}) if every element of $L$ has a TA-factorization.

\medskip
Next, we study C-lattices whose compact elements have a factorization into TA-elements. We call them {\it CTAFLs}. We also explore the C-lattices whose principal elements have a factorization into TA-elements, called {\it PTAFLs}. Clearly, every TAFL is a CTAFL and every CTAFL is a PTAFL.
	
\begin{example}\label{Example 3.1}
As a simple example, it is clear that each prime element is a TA-element. In particular, every ZPI-lattice is a TAFL. By \cite[Example 2.1]{MuAhDu18}, we have that the lattice of ideals of $\mathbb{Z}[\sqrt{-7}]$ is not a TAFL.
\end{example}

Next, we summarize some basic results related to TAFLs in the following remark.

\begin{remark}\label{Remark 3.2}
Let $L$ be a TAFL, let $L_1$ and $L_2$ be C-lattices, let $r\in L$ and let $S\subseteq L_*$ be multiplicatively closed.
\begin{enumerate}
\item[(1)] $\min(x)$ is finite for each $x\in L$.
\item[(2)] $L_1\times L_2$ is a TAFL if and only if $L_1$ and $L_2$ are both TAFLs.
\item[(3)] $L/r$ is a TAFL.
\item[(4)] $L_S$ is a TAFL.
\end{enumerate}
\end{remark}

\begin{proof}
(1) Let $x=\prod_{k=1}^n x_k$ be a TA-factorization of $x$. By \cite[Theorem 3]{JaTeYe14}, we have that $\min(x_i)$ is finite. Then $\min(x)$ is finite, since $\min(x)\subseteq\bigcup_{i=1}^n\min(x_i)$.

\medskip
(2) We infer by \cite[Theorem 2.20]{CaTeYe15} that $(p_1,p_2)$ is a TA-element of $L$ if and only if one of the following conditions is satisfied.

\begin{enumerate}
\item[(a)] $p_1=1_{L_1}$ and $p_2$ is a TA-element of $L_2$.
\item[(b)] $p_2=1_{L_2}$ and $p_1$ is a TA-element of $L_1$.
\item[(c)] $p_1$ and $p_2$ are prime elements of $L_1$ and $L_2$, respectively.
\end{enumerate}

The rest now follows easily.

\medskip
(3) Let $y\in L/r$. Then $y\in L$. By the assumption, $y=\prod_{i=1}^n x_i$ where $x_j$ is a TA-element of $L$ for each $j\in [1,n]$. Note that $x_j$ is a TA-element of $L/r$ for each $j\in [1,n]$ by Remark~\ref{Remark 2.15} and $y=(\prod_{k=1}^n x_i)\vee r=\bigcirc_{i=1}^n x_i$. Consequently, $L/r$ is a TAFL.

\medskip
(4) Recall that if $x$ is a TA-element of $L$, then $x_S=1$ or $x_S$ is a TA-element of $L_S$ by Remark~\ref{Remark 2.16}(2). Let $a\in L_S$. Then we have that $a=y_S$ for some $y\in L$. By assumption, $y$ has a TA-factorization in $L$, meaning $y$ is the finite product of some TA-elements $(y_k)_{k=1}^n$. We can assume without restriction that $(y_k)_S\not=1$ for each $k\in [1,n]$. We have that $a=y_S=(\prod_{k=1}^n y_k)_S=\bigcirc_{k=1}^n (y_k)_S$. This completes the proof.
\end{proof}

\begin{lemma}\label{Lemma 3.3}
Let $L$ be a C-lattice and let $x\in L$ be such that $\sqrt{x}\in\max(L)$. If $x$ has a TA-factorization, then $x\leq (\sqrt{x})^2$ or $(\sqrt{x})^2\leq x$.
\end{lemma}

\begin{proof} Let $x$ have a TA-factorization, let $(\sqrt{x})^2\nleq x$ and set $m=\sqrt{x}$. By \cite[Theorem 3]{JaTeYe14}, $x$ is not a TA-element. By assumption, $x=\prod_{i=1}^n x_i$ where $x_i$ is a TA-element of $L$ for each $i\in [1,n]$ and $n\geq 2$. Since $x\leq x_i$ for each $i\in [1,n]$ and $\sqrt{x}=m\in\max(L)$, we have that $\sqrt{x_i}=m$ for each $i\in [1,n]$. Consequently, $x\leq m^2$.
\end{proof}

\begin{proposition}\label{Proposition 3.4}
Let $L$ be a principally generated TAFL. Then $\dim(L)\leq 1$.
\end{proposition}

\begin{proof}
Observe that since $\dim(L)=\sup\{\dim(L_q)\mid q\in L$ is a prime element$\}$, we can assume (by Remark~\ref{Remark 3.2}(4)) without restriction that $L$ is quasi-local with maximal element $m\neq 0$. It remains to show that each nonmaximal prime element of $L$ is a minimal prime element. Let $p\in L$ be a nonmaximal prime element. Since $L$ is principally generated, there is some principal element $y\in L$ such that $y\leq m$ and $y\nleq p$.

\medskip
Since $L$ is a C-lattice, there is some $q\in\min (p\vee y)$ such that $q\leq m$. Note that $p_q$ is a prime element of $L_q$, $y_q$ is a principal element of $L_q$ and $q_q\in\min((p\vee y)_q)$. If $p_q$ is a minimal prime element of $L_q$, then $p$ is a minimal prime element of $L$. For these reasons, we can assume without restriction that $m\in\min(p\vee y)$. Since $L$ is quasi-local, this implies that $\sqrt{p\vee y}=m$. Next we verify the following claims.

\medskip
Claim 1: $m\neq m^2$.

\medskip
Claim 2: $q\leq m^2$ for every prime element $q<m$.

\medskip
Assume the contrary of claim 1 that $m=m^2$. Clearly, $p\vee y=\prod_{i=1}^k x_i$ where $k$ is a positive integer and $x_i$ is a TA-element of $L$ for each $i\in [1,k]$. Note that $m^2\leq x_i$ for each $i\in [1,k]$ by \cite[Lemma 2]{JaTeYe14}. Therefore, we get that $m^{2k}=m\leq p\vee y\leq m^k=m$ by \cite[Theorem 3]{JaTeYe14} and Lemma~\ref{Lemma 3.3}. Similarly, we have that $m^{2k}=m\leq p\vee y^2\leq m^k=m$. This implies that $p\vee y=m=p\vee y^2$. Observe that $((p\vee y^2):y)=(p:y)\vee y=p\vee y$. Then $1=((p\vee y):y)=((p\vee y^2):y)=p\vee y$, and hence $1=p\vee y=m$, a contradiction.

\medskip
To show that the second claim is true, assume that there is a prime element $q<m$ with $q\nleq m^2$. Clearly, there is a principal element $b\in L$ with $b\leq m$ and $b\nleq q$. Note that $b^n\nleq q$ for each $n\in\mathbb{N}$. Since $L$ is a TAFL and $q\nleq m^2$, we have that $q\vee b^3$ is a TA-element of $L$. It follows that $b^2\leq q\vee b^3$, since $b^3\leq q\vee b^3$. Note that $1=((q\vee b^3):b^2)=(q:b^2)\vee b$. We conclude that $(q:b^2)=1$ or $b=1$. Therefore, $b^2\leq q$ or $b=1$, a contradiction. We infer that every prime element $q\in L$ with $q<m$ satisfies $q\leq m^2$.

\medskip
We will return to the proof of the main part. Since $m\neq m^2$, there is a proper principal element $c\in L$ with $c\nleq m^2$. By claim 2, it follows that $\sqrt{c}=m$. By Lemma~\ref{Lemma 3.3}, we have that $m^2\leq c$. Let $s\in L$ be compact such that $s\leq p$. It follows that $s=\prod_{i=1}^n y_i$ where $y_j$ is a TA-element of $L$ for each $j\in [1,n]$. Since $p$ is prime, there is some $j\in [1,n]$ such that $y_j\leq p$. Since $y_j\leq p<m^2\leq c$ and $c$ is weak meet principal, there is some $\ell\in L$ such that $y_j=c\ell$. Note that $\ell\leq p$ because $c\nleq p$. Consequently, $\ell\leq p\leq m^2\leq c$, and thus $\ell=ct$ for some $t\in L$. Therefore, $y_j=c^2t$, and hence $\ell=ct\leq y_j$, since $y_j$ is a TA-element and $c^2\nleq y_j$. We obtain that $y_j=\ell$, and so $y_j=cy_j$. Therefore, $s=sc$. We have that $s=sc\leq sm\leq s$, and hence $s=sm$. We conclude that $s=0$ by \cite[Theorem 1.4]{An76}, a contradiction. This implies that $p=0$. 
\end{proof}

\begin{theorem}\label{Theorem 3.5}
If $L$ is a TAFL and a Pr\"ufer lattice domain, then $L$ is a ZPI-lattice domain.
\end{theorem}

\begin{proof}
Let $L$ be a Pr\"ufer lattice domain such that $L$ is also a TAFL. Recall from \cite[Theorem 10]{JaTeYe14} that the following statements are equivalent: (1) $p$ is a TA-element, (2) $p$ is a prime element of $L$ or $p=p_1^2$ is a $p_1$-primary element of $L$ or $p=p_1\land p_2$ where $p_1$ and $p_2$ are some nonzero prime elements of $L$. It is sufficient to show that every TA-element of $L$ is a finite product of prime elements of $L$. (Then $L$ is clearly a ZPI-lattice domain.) In view of the equivalence before, we only need to show that if $p,q\in L$ are prime elements such that $p\nleq q$ and $q\nleq p$, then $p\wedge q=pq$. Let $p,q\in L$ be prime elements such that $p\nleq q$ and $q\nleq p$. Assume that $p\vee q$ is proper. Then $p\vee q\leq m$ for some maximal element $m\in L$. Since $L$ is a Pr\"ufer lattice domain, we have that $L_m$ is a totally ordered lattice, and hence $p=p_m\leq q_m=q$ or $q=q_m\leq p_m=p$, a contradiction. We infer that $p\vee q=1$. Set $c=p\wedge q$. Then $pq\leq c$ and $c=c(p\vee q)=cp\vee cq$. Since $c\leq q$, we have that $cp\leq pq$. It follows by analogy that $cq\leq pq$. Therefore, $c=cp\vee cq\leq pq$, and thus $p\wedge q=pq$.
\end{proof}

Next we study CTAFLs.

\begin{remark}\label{Remark 3.6}
Let $L$ be a CTAFL, let $L_1$ and $L_2$ be C-lattices, let $r\in L_*$ and let $S\subseteq L_*$ be multiplicatively closed.
\begin{enumerate}
\item[(1)] $\min(x)$ is finite for each $x\in L_*$.
\item[(2)] $L_1\times L_2$ is a CTAFL if and only if $L_1$ and $L_2$ are both CTAFLs.
\item[(3)] $L/r$ is a CTAFL.
\item[(4)] $L_S$ is a CTAFL.
\end{enumerate}
\end{remark}

\begin{proof}
(1) This can be proved along the same lines as in the proof of Remark~\ref{Remark 3.2}(1).

\medskip
(2) Note that $(L_1\times L_2)_*=\{(a,b)\mid a\in (L_1)_*,b\in (L_2)_*\}$. It follows from \cite{JaTeYe14} that $(p_1,p_2)$ is a TA-element of $L$ if and only if $p_1=1_{L_1}$ and $p_2$ is a TA-element of $L_2$ or $p_2=1_{L_2}$ and $p_1$ is a TA-element of $L_1$ or $p_1$ and $p_2$ are prime elements of $L_1$ and $L_2$, respectively. The rest is straightforward.

\medskip
(3) Observe that every element $a\in L$ with $a\geq r$ is compact in $L$ if and only if $a$ is compact in $L/r$. Now, let $y\in L/r$ be compact. Then $y\geq r$. By the assumption, $y=\prod_{k=1}^n x_k$ where $x_i$ is a TA-element of $L$. Note that $x_i$ is a TA-element of $L/r$ by Remark~\ref{Remark 2.15}(2). Then we get that $y=(\prod_{k=1}^n x_k)\vee r=\bigcirc_{k=1}^n x_k$. Therefore, $L/r$ is a CTAFL.

\medskip
(4) Let $z\in (L_S)_*$. Then $z=y_S$ for some $y\in L_*$. The rest can be shown along the same lines as in the proof of Remark~\ref{Remark 3.2}(4).
\end{proof}

\begin{proposition}\label{Proposition 3.7}
Let $L$ be a principally generated CTAFL that satisfies the ascending chain condition on prime elements. Then $\dim(L)\leq 1$.
\end{proposition}

\begin{proof}
Observe that for each prime element $q\in L$, we have that $L_q$ is a principally generated CTAFL (by Remark~\ref{Remark 3.6}(4)) that satisfies the ascending chain condition on prime elements. Since $\dim(L)=\sup\{\dim(L_q)\mid q\in L$ is a prime element$\}$, we can assume without restriction that $L$ is quasi-local with maximal element $m\neq 0$. It remains to show that each nonmaximal prime element of $L$ is a minimal prime element. Let $p\in L$ be a nonmaximal prime element. Since $L$ is principally generated, there is some principal element $y\in L$ such that $y\leq m$ and $y\nleq p$.

\medskip
Since $L$ is a C-lattice, there exists some $q\in\min(p\vee y)$ such that $q\leq m$. Clearly, $L_q$ is a principally generated CTAFL (by Remark~\ref{Remark 3.6}(4)) with maximal element $q_q$ that satisfies the ascending chain condition on prime elements. Moreover, $p_q$ is a prime element of $L_q$, $y_q$ is a principal element of $L_q$ and $q_q\in\min((p\vee y)_q)$. If $p_q$ is a minimal prime element of $L_q$, then $p$ is a minimal prime element of $L$. For these reasons, we can assume without restriction that $m\in\min(p\vee y)$. Since $L$ is quasi-local, we infer that $\sqrt{p\vee y}=m$. Next we verify the following claims.

\medskip
Claim 1: $m\neq m^2$.

\medskip
Claim 2: $q\leq m^2$ for every prime element $q\in L$ with $q<m$.

\medskip
First we prove claim 1. It follows from Remark~\ref{Remark 3.6}(1) that $\min(x)$ is finite for each compact element $x\in L$. Since $L$ satisfies the ascending chain condition on prime elements, it follows from \cite[Theorem 2]{Ja03} that $p=\sqrt{d}$ for some compact element $d\in L$. Observe that $m=\sqrt{d\vee y}$ and $d\vee y$ is compact. Consequently, $d\vee y=\prod_{i=1}^t x_i$ for some positive integer $t$ and some TA-elements $x_i\in L$. Clearly, $\sqrt{x_i}=m$ for each $i\in [1,t]$, and thus $m^2\leq x_i\leq m$ by \cite[Theorem 3]{JaTeYe14} for each $i\in [1,t]$. Assume to the contrary that $m=m^2$. Then $d\vee y=m$, and hence $p\vee y=m$. We infer that $m=m^2\leq p\vee y^2\leq p\vee y\leq m$, and thus $p\vee y=p\vee y^2$. Since $y$ is (join) principal, we obtain that $1=((p\vee y):y)=((p\vee y^2):y)=p\vee y=m$, a contradiction. This implies that $m\not=m^2$.\qed(Claim 1)

\medskip
Now we prove claim 2. Assume that there is a prime element $q\in L$ such that $q<m$ and $q\nleq m^2$. Since $L$ is principally generated, there is a principal element $b\in L$ such that $b\leq m$ and $b\nleq q$. Note that $b^2\nleq q$. Since $q\nleq m^2$ and $L$ is a C-lattice, there is a compact element $a\in L$ such that $a\leq q$ and $a\nleq m^2$. Since $a$ and $b$ are compact, we have that $a\vee b^3$ is compact. Note that $a\vee b^3$ is a TA-element, since $L$ is a CTAFL and $a\nleq m^2$. Since $b^3\leq a\vee b^3$ and $a\vee b^3$ is a TA-element, we get that $b^2\leq a\vee b^3$. Note that $1=((a\vee b^3):b^2)=(a:b^2)\vee b$. Since $b\leq m$ and $L$ is quasi-local, we have that $(a:b^2)=1$. Consequently, $b^2\leq a\leq q$, a contradiction.\qed(Claim 2)

\medskip
It is sufficient to show that $p=0$. (Then $p$ is a minimal prime element of $L$ and we are done.) Since $m\neq m^2$ by claim 1 and $L$ is principally generated, there is a principal element $c\in L$ such that $c\leq m$ and $c\nleq m^2$. By claim 2, we have that $\sqrt{c}=m$. Furthermore, Lemma~\ref{Lemma 3.3} implies that $m^2\leq c$. Let $s\in L$ be compact such that $s\leq p$. It follows that $s=\prod_{i=1}^k y_i$ for some positive integer $k$ and some TA-elements $y_i\in L$. Obviously, there is some $j\in [1,k]$ such that $y_j\leq p$. Since $y_j\leq p\leq m^2\leq c$ and $c$ is (weak meet) principal, we infer that $y_j=c\ell$ for some $\ell\in L$. Note that $\ell\leq p$, since $c\nleq p$. Consequently, $\ell\leq p\leq m^2\leq c$, and thus $\ell=ct$ for some $t\in L$. This implies that $y_j=c^2t$, and hence $\ell=ct\leq y_j$ (since $y_j$ is a TA-element of $L$ and $c\nleq p$). Therefore, $y_j=\ell$, and thus $y_j=cy_j$. We infer that $s=sc$, and hence $s=sc\leq sm\leq s$. We conclude that $s=sm$. It is an immediate consequence of \cite[Theorem 1.4]{An76} that $s=0$. Finally, we have that $p=0$ (since $L$ is a C-lattice).
\end{proof}

Next we study PTAFLs. We start with a simple observation.

\begin{remark}\label{Remark 3.8}
Let $L$ be a Pr\"ufer lattice. Then $L$ is a CTAFL if and only if $L$ is a PTAFL.
\end{remark}

\begin{proof}
This is obvious, since every compact element in a Pr\"ufer lattice is principal.
\end{proof}

Note that Proposition~\ref{Proposition 3.7} does not hold for PTAFLs. To show that, we consider the following example.

\begin{example}\label{Example 3.9}
Note that if $L$ is the lattice of ideals of a local two-dimensional unique factorization domain $D$ (e.g. take $D=K[X,Y]_{(X,Y)}$ where $K$ is a field and $X$ and $Y$ are indeterminates over $K$), then $L$ is a quasi-local principally generated PTAFL that satisfies the ascending chain condition on prime elements and $\dim(L)=2$.
\end{example}

\begin{theorem}\label{Theorem 3.10}
Let $(L,m)$ be a quasi-local principally generated C-lattice domain such that $m^2$ is comparable and $\bigwedge_{n\in\mathbb{N}} m^n=0$. Then $L$ is a TAFL.
\end{theorem}

\begin{proof}
Assume that $m=m^2$. Since $L$ is a lattice domain and $\bigwedge_{n\in\mathbb{N}} m^n=0$, we infer that $m=0$. Therefore, $L$ is a field and hence we get the desired properties. Now, suppose that $m\neq m^2$. Then there is a proper principal element $x\in L$ with $x\nleq m^2$. By the assumption, we get that $m^2<x$. Since $x$ is meet principal, we conclude that $m^2=xm$.

Next we show that $\dim(L)\leq 1$. Let $p\in L$ be a prime element such that $p<m$. Clearly, $m^2\nleq p$, and thus $p\leq m^2=xm\leq x$. Observe that $p=xb$ for some $b\in L$, and hence $b\leq p$ (since $x\nleq p$). This implies that $p=xp$, and so $p=x^np\leq m^n$ for each $n\in\mathbb{N}$. Therefore, $p\leq\bigwedge_{n\in\mathbb{N}} m^n=0$, and thus $p=0$. This shows that $\dim(L)\leq 1$.

By \cite[Theorem 3]{JaTeYe14}, we obtain that $x$ is a TA-element. Let $z\in L$ be proper. It is clear that $z=0$ is a TA-element. Now let $z\not=0$. If $m^2\leq z$, then the proof is complete by \cite[Theorem 3]{JaTeYe14}. Now let $z\leq m^2$. Let $n$ be the largest positive integer satisfying $z\leq m^n$. We conclude that $z\leq m^n=x^{n-1}m$. Consequently, $z\leq x^{n-1}$. Note that $z=x^{n-1}a$ for some $a\in L$, since $x^{n-1}$ is principal. If $a\leq m^2$, then we obtain $z=x^{n-1}a\leq x^{n-1}m^2=m^{n+1}$, leading to a contradiction. This implies that $m^2\leq a$, and thus $a$ is a TA-element. Therefore, $z$ has a TA-factorization.
\end{proof}

\begin{theorem}\label{Theorem 3.11}
Let $(L,m)$ be a quasi-local principally generated C-lattice domain. Then $L$ is a TAFL if and only if $\dim(L)\leq 1$ and $L$ is a PTAFL. If these equivalent conditions are satisfied, then $\bigwedge_{n\in\mathbb{N}} m^n=0$ and $m^2$ is comparable.
\end{theorem}

\begin{proof}
$(\Rightarrow)$ This follows from Proposition~\ref{Proposition 3.4}.

\medskip
$(\Leftarrow)$ Let $L$ be a PTAFL such that $\dim(L)\leq 1$. First, assume that $m=m^2$. Since $\dim(L)\leq 1$, then $m$ is the only TA-element whose radical is $m$ by \cite[Lemma 2]{JaTeYe14}. We infer that each nonzero proper principal element of $L$ is equal to $m$. This implies that $m$ is the only nonzero proper element of $L$ and we are done.

\medskip
Now let $m\neq m^2$. Then there is a principal element $x\in L$ such that $x\leq m$ and $x\nleq m^2$. Since $x\nleq m^2$, we infer that $x$ is a TA-element (since $x$ cannot be the product of more than one TA-element). From \cite[Theorem 3]{JaTeYe14}, we get that $m^2\leq x$ since $\sqrt{x}=m$. Since $x$ is a (weak meet) principal element, we conclude that $m^2=xm$. Let $z\in L$ be proper. We have to show that $z$ has a TA-factorization. If $z=0$, then we are done, since $L$ is a lattice domain. Therefore, we can assume without restriction that $z\neq 0$.

Next we show that $\bigwedge_{n\in\mathbb{N}} m^n=0$. Assume the contrary that $\bigwedge_{n\in\mathbb{N}} m^n\neq 0$. Note that each nonzero TA-element $v\in L$ satisfies $m^2\leq v$ by \cite[Lemma 2]{JaTeYe14}. Clearly, there is some nonzero principal element $x\in L$ such that $x\leq\bigwedge_{n\in\mathbb{N}} m^n$. We have that $x=\prod_{i=1}^k a_i$ where $a_i$ is a TA-element for each $i\in [1,k]$. We obtain that $(m^2)^k\leq x\leq (m^4)^k\leq (m^2)^k$, and hence $x=x^2$. Since $x$ is principal, we infer that $1=x\vee (0:x)$. Therefore, $x=1$, a contradiction.

By Theorem~\ref{Theorem 3.10}, it remains to show that $m^2$ is comparable. Let $y\in L$ be proper such that $y\nleq m^2$. Since $L$ is principally generated, there is some proper principal element $x\in L$ such that $x\leq y$ and $x\nleq m^2$. Observe that $x$ is a TA-element (since $L$ is a PTAFL). From \cite[Lemma 2]{JaTeYe14}, we get that $m^2\leq x$ since $\dim(L)\leq 1$.
\end{proof}

\begin{proposition}\label{Proposition 3.12}
Let $(L,m)$ be a quasi-local principally generated C-lattice such that $L$ is not a lattice domain. Then $L$ is a TAFL if and only if $\dim(L)=0$ and $L$ is a PTAFL.
\end{proposition}

\begin{proof}
$(\Rightarrow)$ Let $L$ be a TAFL. By the proof of Proposition~\ref{Proposition 3.4}, we know that if $p\in L$ is a nonmaximal prime element of $L$, then $p=0$. Therefore, $\dim(L)\leq 1$. Since $L$ is not a lattice domain, we infer that $\dim(L)=0$.

\medskip
$(\Leftarrow)$ Let $\dim(L)=0$ and let $L$ be a PTAFL. Then $0=\prod_{i=1}^k x_i$ where $k\in\mathbb{N}$ and $x_j\in L$ are TA-elements with $\sqrt{x_j}=m$ for each $j\in [1,k]$. By \cite[Lemma 2]{JaTeYe14}, we have that $m^2\leq x_i$ for each $i\in [1,k]$, and thus $(m^2)^k\leq 0$. This implies that $0=m^{2k}$. If $m=m^2$, then $m=0$, a contradiction (since $L$ is not a domain). Consequently, $m\neq m^2$, and hence there is a principal element $z\in L$ such that $z\leq m$ and $z\nleq m^2$. Since $z\nleq m^2$, we infer that $z$ is a TA-element (since $z$ cannot be the product of more than one TA-element). From \cite[Lemma 2]{JaTeYe14}, we get that $m^2\leq z$ (since $\sqrt{z}=m$). Since $z$ is a (weak meet) principal element, we conclude that $m^2=zm$.

Next we show that $m^2$ is comparable. Let $y\in L$ be proper such that $y\nleq m^2$. Then there is some proper principal element $v\in L$ such that $v\leq y$ and $v\nleq m^2$. Since $L$ is a PTAFL, we have that $v$ is a TA-element, and hence $m^2=(\sqrt{v})^2\leq v$ by \cite[Lemma 2]{JaTeYe14}.

Finally, let $x\in L$ be nonzero. If $m^2\leq x$, then $x$ is a TA-element because it is a primary element as shown by \cite[Lemma 5]{JaTeYe14}. Now let $x\leq m^2$. Then $x\leq z$. Note that $x\nleq 0=z^{2k}$. We infer that $x\leq z^r$ and $x\nleq z^{r+1}$ for some $r\in\mathbb{N}$. Consequently, $x=z^ra$ for some $a\in L$ with $a\nleq z$. It follows that $a\nleq m^2$, and thus $m^2\leq a$. Observe that $a$ is a TA-element by Proposition~\ref{Proposition 2.7}. Therefore, $x$ is a finite product of TA-elements of $L$.
\end{proof}

\begin{remark}\label{Remark 3.13}
Let $(L,m)$ be a quasi-local principal element TAFL domain. Then every proper element of $L$ is a power of $m$.
\end{remark}

\begin{proof}
We know that a TA-element of $L$ equals $m$ or $m^2$ by \cite[Theorem 9]{JaTeYe14}. This completes the proof.
\end{proof}

\section[OAFLs]{OAFLs and their generalizations}\label{4}

In this section, we study the factorization of elements of $L$ with respect to the OA-elements, similar to the previous section. It consists of three parts. The first part involves C-lattices whose elements possess an OA-factorization, called {\it OAFLs}. Next, we examine C-lattices whose compact elements have a factorization into OA-elements, called {\it COAFLs}. Finally, we explore the C-lattices whose principal elements have a factorization into OA-elements, called {\it POAFLs}. It can easily be shown that every OAFL is a COAFL and every COAFL is a POAFL. We continue by presenting some results related to OAFLs.

\begin{example}\label{Example 4.1}
Let $L$ be the lattice of ideals of $\mathbb{Z}[2i]$. Note that $(2+2i)$ has no OA-factorization. For a more general example, consider the lattice $L$ of ideals of a non-integrally closed, non-quasi-local domain. It is a simple consequence of \cite[Theorem 46.7]{Gi92} that $L$ is not a POAFL.	
\end{example}

\begin{remark}\label{Remark 4.2}
Let $L$ be an OAFL, let $r\in L$ and let $S\subseteq L_*$ be multiplicatively closed.
\begin{enumerate}
\item[(1)] $L$ is both a Q-lattice and a TAFL.
\item[(2)] $\min(x)$ is finite for each $x\in L$.
\item[(3)] $L/r$ is an OAFL.
\item[(4)] $L_S$ is an OAFL.
\end{enumerate}
\end{remark}

\begin{proof}
\medskip
(1) Since every OA-element of $L$ is a TA-element and a primary element, we have that $L$ is both a Q-lattice and a TAFL.

\medskip
(2) This follows from (1) and Remark~\ref{Remark 3.2}(1).

\medskip
(3) This follows along similar lines as in Remark~\ref{Remark 3.2}(3).

\medskip
(4) Here we can mimic the proof of Remark~\ref{Remark 3.2}(4).
\end{proof}

\begin{corollary}\label{Corollary 4.3}
Let $L$ be a principally generated OAFL. Then $\dim(L)\leq 1$.
\end{corollary}

\begin{proof}
This is an immediate consequence of Remark~\ref{Remark 4.2}(1) and Proposition~\ref{Proposition 3.4}.
\end{proof}

\begin{proposition}\label{Proposition 4.4}
Let $(L,m)$ be a quasi-local principally generated C-lattice such that $m^2$ is comparable and $(m$ is nilpotent or $L$ is a lattice domain with $\bigwedge_{n\in\mathbb{N}} m^n=0)$. Then $L$ is an OAFL and every principal element of $L$ is a finite product of principal OA-elements.
\end{proposition}

\begin{proof}
It is an immediate consequence of Theorem~\ref{Theorem 3.11} and Proposition~\ref{Proposition 3.12} that $L$ is a TAFL, and thus $\dim(L)\leq 1$ by Proposition~\ref{Proposition 3.4}. Now it is clear that for each proper nonprime element $x\in L$, it follows that $\sqrt{x}=m$. Therefore, every TA-element of $L$ is an OA-element by Proposition~\ref{Proposition 2.14}, and hence $L$ is an OAFL. It follows by analogy from Theorem~\ref{Theorem 3.11} and Propositions~\ref{Proposition 2.14} and~\ref{Proposition 3.12} and their proofs together with \cite[Theorem 9]{AnJo96} that every principal element of $L$ is a finite product of principal OA-elements.
\end{proof}

Next we study COAFLs.

\begin{remark}\label{Remark 4.5}
Let $L$ be a COAFL, let $r\in L_*$ and let $S\subseteq L_*$ be multiplicatively closed.
\begin{enumerate}
\item[(1)] $L$ is a CTAFL.
\item[(2)] $\min(x)$ is finite for each $x\in L_*$.
\item[(3)] $L/r$ is a COAFL.
\item[(4)] $L_S$ is a COAFL.
\end{enumerate}
\end{remark}

\begin{proof}
(1) This is clear, since every OA-element of $L$ is a TA-element.

\medskip
(2) This follows from (1) and Remark~\ref{Remark 3.6}(1).

\medskip
(3) Here one can mimic the proof of Remark~\ref{Remark 3.6}(3).

\medskip
(4) This can be proved along similar lines as in Remark~\ref{Remark 3.6}(4).
\end{proof}

\begin{proposition}\label{Proposition 4.6}
Let $(L,m)$ be a quasi-local C-lattice and let $G\subseteq L$ be such that for each $x\in L$, there is some $F\subseteq G$ with $x=\bigvee F$. Suppose that every $g\in G$ has an OA-factorization. Then each minimal nonmaximal prime element of $L$ is a weak meet principal element.
\end{proposition}

\begin{proof}
Let $p\in L$ be a minimal nonmaximal prime element. First we show that $x=p(x:p)$ for each element $x\in G$ with $x\leq p$. Let $x\in G$ with $x\leq p$. By the assumption, we can write $x=\prod_{i=1}^n x_i$ where $n$ is a positive integer and $x_i$ is an OA-element of $L$ for each $i\in [1,n]$. Since $x\leq p$, we have that $x_i\leq p$ for some $i\in [1,n]$. If $x_i$ is not a prime element, then $m^2\leq x_i\leq p$ by Proposition~\ref{Proposition 2.7}. We infer that $m=p$, a contradiction. Consequently, $x_i$ is a prime element, and hence $x_i=p$. We have that $x=pz$, where $z=\prod_{k=1,k\not=i}^n x_k$. Observe that $x=p(x:p)$.

\medskip
Let $y\in L$ be such that $y\leq p$. We have that

\[
y=\bigvee\{v\in G\mid v\leq y\}=\bigvee\{p(v:p)\mid v\in G,v\leq y\}=p\bigvee\{(v:p)\mid v\in G,v\leq y\}.
\]

Set $w=\bigvee\{(v:p)\mid v\in G,v\leq y\}$. Then $w\in L$ and $y=pw$, and thus $p$ is weak meet principal.
\end{proof}

\begin{corollary}\label{Corollary 4.7}
Let $(L,m)$ be a quasi-local COAFL. Then each minimal nonmaximal prime element of $L$ is a weak meet principal element.
\end{corollary}

\begin{proof}
This follows from Proposition~\ref{Proposition 4.6}.
\end{proof}

\begin{corollary}\label{Corollary 4.8}
Let $(L,m)$ be a quasi-local principally generated POAFL. Then every minimal nonmaximal prime element of $L$ is a principal element.
\end{corollary}

\begin{proof}	
It is clear from Proposition~\ref{Proposition 4.6} that every minimal nonmaximal prime element of $L$ is a weak meet principal element. Therefore, every minimal nonmaximal prime element of $L$ is a principal element by \cite[Theorem 1.2]{An76}.
\end{proof}

Next we study POAFLs. We start with a simple observation.

\begin{remark}\label{Remark 4.9}
Let $L$ be a Pr\"ufer lattice. Then $L$ is a COAFL if and only if $L$ is a POAFL.
\end{remark}

\begin{proof}
This is obvious, since every compact element in a Pr\"ufer lattice is principal.
\end{proof}

By \cite[Corollary 4.10]{AnJa96} we obtain that a non-quasi-local domain has a COAFL ideal lattice if and only if it is Dedekind.

\begin{lemma}\label{Lemma 4.10}
Let $(L,m)$ be a quasi-local principally generated C-lattice. If the join of any two principal elements of $L$ has an OA-factorization, then every nonmaximal prime element of $L$ is a principal element and $\dim(L)\leq 2$.
\end{lemma}

\begin{proof}
Let the join of any two principal elements of $L$ have an OA-factorization. First, we show that $\dim(L)\leq 2$. Let $p\in L$ be a nonmaximal prime element of $L$. Note that $L_p$ is quasi-local and $L_p$ is generated by the set of elements $\Omega:=\{a_p\mid a\in L$ is principal$\}$. Now we prove that the join of any two elements of $\Omega$ has a prime factorization. Let $z\in L_p$ be proper such that $z$ is the join of two elements of $\Omega$. Then there are some principal elements $a,b\in L$ such that $a\vee b\leq p$ and $z=a_p\vee_p b_p$. Obviously, there are some $n,m\in\mathbb{N}$ and OA-elements $(x_i)_{i=1}^m$ of $L$ such that $n\leq m$, $x_j\leq p$ for each $j\in [1,n]$, $x_j\nleq p$ for each $j\in [n+1,m]$ and $a\vee b=\prod_{i=1}^m x_i$. We infer by Proposition~\ref{Proposition 2.7} that $x_i$ is a prime element of $L$ for each $i\in [1,n]$. This implies that $z=a_p\vee_p b_p=(a\vee b)_p=(\prod_{i=1}^{n} x_i)_p=\bigcirc_{i=1}^{n} (x_i)_p$ where $(x_j)_p$ is a prime element of $L_p$ for each $j\in [1,n]$. Consequently, $L_p$ is a ZPI-lattice by \cite[Theorem 8]{Ja02}, and thus $\dim(L_p)\leq 1$ by \cite[Theorem 2.6]{AnJa96}. Therefore, $\dim(L)\leq 2$.

\medskip
Let $q\in L$ be a nonmaximal prime element. If $q$ is minimal, then it follows from Corollary~\ref{Corollary 4.8} that $q$ is principal. Now let $q$ be not minimal. Assume that $qm=q$.  There is a $p\in\min(0)$ such that $p<q$, and thus there is a principal element $c\in L$ such that $c\nleq p$ and $c\leq q$. By assumption, $c\vee p$ has an OA-factorization, since $p$ is principal by Corollary~\ref{Corollary 4.8}. Since $\dim(L)\leq 2$, we have that $q$ is minimal over $c\vee p$. Observe that $c\vee p=\prod_{i=1}^n x_i$ where $x_i$ is an OA-element for each $i\in [1,n]$. We get that $x_j\leq q$ for some $j\in [1,n]$. Note that $m^2\nleq x_j$, and hence $x_j=q$ by Proposition~\ref{Proposition 2.7}. We infer that $c\vee p=q\ell$ for some $\ell\in L$. We conclude that $(c\vee p)m=c\vee p$ by the assumption, and thus $c\vee p=0$ by \cite[Theorem 1.4]{An76}. Consequently, $c=p=0$, which contradicts the fact that $c\nleq p$.

\medskip
It follows that $q\neq qm$. Since $L$ is principally generated, there is a principal element $a\in L$ such that $a\leq q$ and $a\nleq qm$. Since $q$ is nonmaximal and $L$ is principally generated, there is a principal element $b\in L$ such that $b\leq m$ and $b\nleq q$. It remains to show that $x\leq a$ for each principal element $x\in L$ such that $x\leq q$. (Then $q=a$ is principal, since $L$ is principally generated and $a\leq q$.) Let $x\in L$ be principal such that $x\leq q$. Note that $xb^2$ is principal, and thus $a\vee xb^2$ is the join of two principal elements of $L$. Consequently, $a\vee xb^2$ has an OA-factorization. Since $a\vee xb^2\leq p$, there are $v,w\in L$ such that $v$ is an OA-element of $L$, $v\leq q$ and $a\vee xb^2=vw$. If $w$ is proper, then $w\leq m$, and hence $a\leq a\vee xb^2\leq qm$, a contradiction. Therefore, $a\vee xb^2=v$ is an OA-element. Assume that $b^2\leq a\vee xb^2$. Since $b^2$ is principal, we have that $1=(a\vee xb^2:b^2)=(a:b^2)\vee x$, and thus $(a:b^2)=1$, since $L$ is quasi-local. Consequently, $b^2\leq a\leq q$, and hence $b\leq q$, a contradiction. This implies that $b^2\nleq a\vee xb^2$. Since $xb^2\leq a\vee xb^2$ and $a\vee xb^2$ is an OA-element, we conclude that $x\leq a\vee xb^2$. By \cite[Theorem 1.4]{An76}, it follows that $x\leq a$.
\end{proof}

\begin{proposition}\label{Proposition 4.11}
Let $L$ be a principally generated C-lattice such that the join of any two principal elements of $L$ has an OA-factorization. Then $\dim(L)\leq 1$ and if $\dim(L)=1$ and $L$ is quasi-local, then $L$ is a domain.
\end{proposition}

\begin{proof} First let every OA-element of $L$ be a prime element. We infer that $L$ is a ZPI-lattice and every element of $L$ is principal by \cite[Theorem 8]{Ja02}. Therefore, $\dim(L)\leq 1$ by \cite[Theorem 2.6]{AnJa96}. For the rest of this paragraph, let $\dim(L)=1$ and let $L$ be quasi-local with maximal element $m$. Clearly, there is a nonmaximal prime element $p\in L$. Since $m$ is principal and $p<m$, we conclude that $p=pm$. Consequently, $p=0$ by \cite[Theorem 1.4]{An76}, and thus $L$ is a domain.

\medskip
Now let there be an OA-element that is not a prime element. Then $L$ is quasi-local by Proposition~\ref{Proposition 2.6}. Let $m$ be the maximal element of $L$. We infer by Proposition~\ref{Proposition 2.7} that $m\neq m^2$. There exists some principal element $c\in L$ such that $c\nleq m^2$ and $c\leq m$.

\medskip
Claim: $p\leq m^2$ for each nonmaximal prime element $p\in L$.

\medskip
Let $p\in L$ be a nonmaximal prime element. By Lemma~\ref{Lemma 4.10}, we have that $\dim(L)\leq 2$. Therefore, we can assume without restriction that there are no prime elements of $L$ that are properly between $p$ and $m$ (i.e., for each prime element $r\in L$ with $p\leq r\leq m$, it follows that $r\in\{p,m\}$).

Next we show that $p=\bigwedge\{p\vee a\mid a\in L$ is principal and $a\nleq p\}$. Assume to the contrary that $p\not=\bigwedge\{p\vee a\mid a\in L$ is principal and $a\nleq p\}$. Then there is a principal element $y\in L$ such that $y\nleq p$ and $y\leq\bigwedge\{p\vee a\mid a\in L$ is principal and $a\nleq p\}$. Since $p<m$, there is a principal element $b\in L$ such that $b\nleq p$ and $b\leq m$. Since $y^2$ is principal and $y^2\nleq p$, we have that $y\leq p\vee y^2$, and hence $1=((p\vee y^2):y)=(p:y)\vee y=p\vee y$. It follows that $1=y\leq p\vee b\leq m$, a contradiction. Therefore, $p=\bigwedge\{p\vee a\mid a\in L$ is principal and $a\nleq p\}$.

Assume that $p\nleq m^2$. It is sufficient to show that $m^2\leq p\vee a$ for each proper principal $a\in L$ with $a\nleq p$. (Then $m^2\leq\bigwedge\{p\vee d\mid d\in L$ is principal and $d\nleq p\}=p$, and hence $p=m$, a contradiction.) Let $a\in L$ be a proper principal element such that $a\nleq p$. Since $p$ is principal by Lemma~\ref{Lemma 4.10}, it follows that $p\vee a$ has an OA-factorization in $L$. Since $p\vee a\nleq m^2$ and $p\vee a$ is proper, we obtain that $p\vee a$ is an OA-element. Clearly, $p\vee a$ cannot be a nonmaximal prime element, and hence $m^2\leq p\vee a$ by Proposition~\ref{Proposition 2.7}. Consequently, $p\leq m^2$.\qed(Claim)

\medskip
It is sufficient to show that $r=0$ for each nonmaximal prime element $r\in L$. Let $r\in L$ be a nonmaximal prime element. Since $c$ has an OA-factorization and $c\nleq m^2$, we infer that $c$ is an OA-element of $L$. It follows by the claim that $c$ cannot be a nonmaximal prime element of $L$. Observe that $r\leq m^2\leq c$ by the claim and by Proposition~\ref{Proposition 2.7}. Since $c$ is a (weak meet) principal element of $L$, $r$ is a prime element of $L$ and $c\nleq r$, we conclude that $r=cr$. Therefore, $r=0$ by \cite[Theorem 1.4]{An76}.
\end{proof}

\begin{proposition}\label{Proposition 4.12}
Let $L$ be a principally generated C-lattice and set $m={\rm J}(L)$. If the join of any two principal elements of $L$ has an OA-factorization, then $L$ satisfies one of the following conditions.
\begin{enumerate}
\item[(a)] $L$ is a ZPI-lattice.
\item[(b)] $L$ is a quasi-local lattice, $m^2$ is comparable and $m$ is a nilpotent element.
\item[(c)] $L$ is a quasi-local lattice domain, $m^2$ is comparable and $\bigwedge_{n\in\mathbb{N}} m^n=0$.
\end{enumerate}
\end{proposition}

\begin{proof}
Let the join of any two principal elements of $L$ have an OA-factorization. Assume that every OA-element of $L$ is prime. By \cite[Theorem 8]{Ja02}, it follows that $L$ is a ZPI-lattice. Now, assume that there is an OA-element which is not a prime. We conclude that $(L,m)$ is quasi-local with $m^2\neq m$ by Propositions~\ref{Proposition 2.6} and~\ref{Proposition 2.7}. There is a proper principal element $x\in L$ such that $x\nleq m^2$. Clearly, $x$ is not the product of more than one OA-element, and thus $x$ is an OA-element. It follows from Proposition~\ref{Proposition 4.11} that $\dim(L)\leq 1$.

\medskip
First let $\dim(L)=0$. We show that $m^2$ is comparable. Let $z\in L$ be proper such that $z\nleq m^2$. There is a principal element $a\in L$ with $a\leq z$ and $a\nleq m^2$. Note that $a$ is an OA-element which is not a prime element. We have that $m^2\leq a$ by Proposition~\ref{Proposition 2.7}, and thus $m^2\leq z$. Consequently, $m^2$ is comparable. Clearly, $0$ has an OA-factorization. This implies that $m^k\leq 0$ for some positive integer $k$, and hence $m^k=0$.

\medskip
Now let $\dim(L)=1$. We obtain that $(L,m)$ is a quasi-local lattice domain by Proposition~\ref{Proposition 4.11}. To verify that $\bigwedge_{n\in\mathbb{N}} m^n=0$, assume the contrary that $\bigwedge_{n\in\mathbb{N}} m^n\neq 0$. There is a nonzero principal element $x\in L$ with $x\leq\bigwedge_{n\in\mathbb{N}} m^n$. Since $x$ is a product of $k$ OA-elements of $L$, we conclude that $m^{2k}\leq x\leq m^{4k}\leq m^{2k}$. In particular, we get that $x=x^2$. Since $x$ is principal, we have that $1=x\vee (0:x)$. We infer that $x=1$, a contradiction. Therefore, $\bigwedge_{n\in\mathbb{N}} m^n=0$. Finally, we show that $m^2$ is comparable. Let $z\in L$ be proper such that $z\nleq m^2$. There is a principal element $a\in L$ with $a\leq z$ and $a\nleq m^2$. We get that $a$ is a nonzero OA-element. Therefore, $m^2\leq a$ (since $\dim(L)=1$), and thus $m^2\leq z$.
\end{proof}	

\begin{theorem}\label{Theorem 4.13}
Let $L$ be a principally generated C-lattice and set $m={\rm J}(L)$. The following statements are equivalent.
\begin{enumerate}
\item[(1)] $L$ is an OAFL.
\item[(2)] $L$ is a COAFL.
\item[(3)] The join of any two principal elements of $L$ has an OA-factorization.
\item[(4)] $L$ satisfies one of the following conditions.
\begin{enumerate}
\item[(a)] $L$ is a ZPI-lattice.
\item[(b)] $L$ is a quasi-local lattice, $m^2$ is comparable and $m$ is a nilpotent element.
\item[(c)] $L$ is a quasi-local lattice domain, $m^2$ is comparable and $\bigwedge_{n\in\mathbb{N}} m^n=0$.
\end{enumerate}
\end{enumerate}
\end{theorem}

\begin{proof}
(1) $\Rightarrow$ (2) This is obvious.

\medskip
(2) $\Rightarrow$ (3) Note that every principal element is compact, and hence the join of each two principal elements is compact. The statement is now immediately clear.

\medskip
(3) $\Rightarrow$ (4) This follows from Proposition~\ref{Proposition 4.12}.

\medskip
(4) $\Rightarrow$ (1) If $L$ is a ZPI-lattice, then clearly $L$ is an OAFL. Now let $L$ be not a ZPI-lattice. It is an immediate consequence of Proposition~\ref{Proposition 4.4} that $L$ is an OAFL.
\end{proof}

\begin{theorem}\label{Theorem 4.14}
Let $L$ be a principally generated C-lattice. The following statements are equivalent.
\begin{enumerate}
\item[(1)] $L$ is a ZPI-lattice.
\item[(2)] $L$ is a Pr\"ufer OAFL.
\item[(3)] $L$ is a Pr\"ufer POAFL.
\end{enumerate}
\end{theorem}

\begin{proof}
(1) $\Rightarrow$ (2) $\Rightarrow$ (3) This follows from \cite[Theorem 8]{Ja02}.

\medskip
(3) $\Rightarrow$ (1) Let $L$ be a Pr\"ufer POAFL. If $L$ is not quasi-local, then the prime elements coincide with the OA-elements. By Theorem~\ref{Theorem 4.13}, we infer that $L$ is a ZPI-lattice. Assume that $(L,m)$ is a quasi-local lattice with maximal element $m$. Then $m^2$ is comparable by Theorem~\ref{Theorem 4.13}. We know from Proposition~\ref{Proposition 2.13}(3) that each OA-element is either prime or equal to $m^2$. Therefore, $L$ is a ZPI-lattice.
\end{proof}

Finally, we provide a theorem that connects the various types of factorization lattices for a quasi-local principally generated C-lattice domain.

\begin{theorem}\label{Theorem 4.15}
Let $(L,m)$ be a quasi-local principally generated C-lattice domain. The following statements are equivalent.
\begin{enumerate}
\item[(1)] $L$ is an OAFL.
\item[(2)] $L$ is a TAFL.
\item[(3)] $L$ is a COAFL.
\item[(4)] $L$ is a CTAFL that satisfies the ascending chain condition on prime elements.
\item[(5)] $\dim(L)\leq 1$ and $L$ is a POAFL.
\item[(6)] $\dim(L)\leq 1$ and $L$ is a PTAFL.
\item[(7)] $m^2$ is comparable and $\bigwedge_{n\in\mathbb{N}} m^n=0$.
\end{enumerate}
\end{theorem}

\begin{proof}
(1) $\Leftrightarrow$ (3) $\Leftrightarrow$ (7) This follows from Theorem~\ref{Theorem 4.13}.

\medskip
(1) $\Rightarrow$ (5) $\Rightarrow$ (6) This follows from Corollary~\ref{Corollary 4.3}.

\medskip
(2) $\Leftrightarrow$ (6) $\Leftrightarrow$ (7) This is an immediate consequence of Theorems~\ref{Theorem 3.10} and~\ref{Theorem 3.11}.

\medskip
(2) $\Rightarrow$ (4) Clearly, $L$ is a CTAFL. Moreover, $\dim(L)\leq 1$ by Proposition~\ref{Proposition 3.4}. It is clear now that $L$ satisfies the ascending chain condition on prime elements.

\medskip
(4) $\Rightarrow$ (6) Obviously, $L$ is a PTAFL. We infer by Proposition~\ref{Proposition 3.7} that $\dim(L)\leq 1$.
\end{proof}

{\bf ACKNOWLEDGEMENTS.} We want to thank the referee for reading our manuscript very carefully and for a multitude of corrections and suggestions that substantially improved the quality of our manuscript. The first-named author was supported by the Austrian Science Fund FWF, Project Number P36742-N. The second-named author was supported by Scientific and Technological Research Council of Turkey (TUBITAK) 2219-International Postdoctoral Research Fellowship Program for Turkish Citizens, (Grant No: 1059B192202920).

\end{document}